\documentclass{amsart}
\usepackage{amscd,amssymb,latexsym}
\usepackage{tikz}
\usepackage{tikz-3dplot}

\begin{document}
\def\a{{\mathbf a}}\def\l{{\lambda}}	
\def\b{{\beta}} 	\def\ra{{\rangle}}	\def\PP{{\mathbb P}}\def\NN{{\mathbb N}}  	\def\ZZ{{\mathbb Z}}  
\def\u{{\mathbf u}}\def\v{{\mathbf v}} \def\S{{\mathcal S}}\def\x{{\mathbf x}}
\def\bb{{\mathbf b}}

\def\F{{\mathcal F}}

\def\sort{\operatorname {sort}}\def\codim{\operatorname {codim}}	
\def\coker{\operatorname {coker}}	
\def\Hom{\operatorname {Hom}}
\def\Int{\operatorname {Int}}
\def\Ker{\operatorname {Ker}}  	
\def\reg{\operatorname {reg}}		   
\def\supp{\operatorname {supp}} 
\def\Tor{\operatorname {Tor}} 	

\newtheorem{Theorem}{Theorem}[section]
\newtheorem{Corollary}[Theorem]{Corollary}
\newtheorem{Lemma}[Theorem]{Lemma}
\newtheorem{Proposition}[Theorem]{Proposition}
\newtheorem{Fact}[Theorem]{Fact}
\newtheorem{Example}[Theorem]{Example}
\newtheorem{Definition}[Theorem]{Definition}
\newtheorem{Remark}[Theorem]{Remark}
\numberwithin{equation}{section}

\title{The Koszul property of pinched Veronese varieties}  
\author{Thanh Vu}
\address{Department of Mathematics, University of California at Berkeley, Berkeley, CA 94720}
\email{vqthanh@math.berkeley.edu}
\date{\today}

\begin{abstract} Let $K$ be an arbitrary field. Let $n,d \ge 2$ be positive integers. Let $V(n,d)$ be the set of all lattice points $\mathbf b = (b_1, ..., b_n)$ in $\NN^n$ such that $\sum_{i=1}^n b_i = d$. Let $\Gamma = V(n,d) \setminus \{ \mathbf a \}$ for some element $\mathbf a  \in V(n,d)$. In this paper we prove that the semigroup ring $K[\Gamma]$ is Koszul unless $d \ge 3$ and ${\mathbf a} = (0, ...,0, 2, d-2)$ or one of its permutations. This generalizes results of Caviglia, Conca, and Tancer.
\end{abstract}
\maketitle

\section{Introduction}
Let $K$ be an arbitrary field. A standard graded $K$-algebra $R$ is Koszul if the residue field $K$ has linear resolution over $R$. See volume \cite{PP} by Polishchuk and Positselki for an extensive treatment with various interesting aspects of Koszul algebras. 

In general, it can be difficult to prove that a given graded algebra is Koszul. For example, for any positive integer $s$, in \cite{R}, Roos gave an example of a quadratic algebra $R$ whose minimal free resolution of the residue field $K$ over $R$ is linear up to $s$ steps, but not linear at the $(s+1)$st step. Thus knowing that the resolution of the residue field $K$ over $R$ is linear up to some finite step does not guarantee that $R$ is Koszul.

In 1993, in conversation with Peeva, Sturmfels asked if the pinched Veronese $K[x^3,x^2y,xy^2,y^3,x^2z,xz^2,y^2z,yz^2,z^3]$ is Koszul. This question remained open until Caviglia answered in affirmative in \cite{C} in 2009, (see also \cite{CC} for a new proof and some generalization of this result).

In this paper we introduce a new method and use it to prove:

\begin{Theorem}\label{main} Let $n, d \ge 2$ be positive integers. Let $V(n,d)$ be the set of all lattice points ${\mathbf b} = (b_1, ..., b_n)$ in $\NN^n$ such that $\sum_{i=1}^n b_i = d$. Let $\Gamma = V(n,d) \setminus \{\a\}$ for some element $\a \in V(n,d)$. The algebra $K[\Gamma]$ is Koszul unless $d \ge 3$ and $\a = (0, ..., 0, 2, d-2)$ or one of its permutations.
\end{Theorem}

From now on, we assume that $a_1 \le a_2 \le \cdots \le a_n$. Note that, in this class of $\Gamma = V(n,d) \setminus \{\a\}$, very little was known about the Koszul property of $K[\Gamma]$. In \cite{S}, Sturmfels showed that $R$ is Koszul when $\a$ is of the form $(0, ..., 0, d)$. The case of Caviglia's theorem is the case when $n =d = 3$ and $\a = (1,1,1)$. Recently, in \cite{T}, Tancer showed the Koszul property of $K[\Gamma]$ when $\Gamma = V(n,n) \setminus \{(1, ..., 1)\}.$ Our result recovers and generalizes work of Cagvilia, Conca and Tancer. Moreover, it provides a new large family of Koszul algebras.

Various techniques have been devised to prove that an algebra $R$ is Koszul. These include proving that $R$ has a quadratic Gr\"obner basis or Koszul filtration. Other techniques involve proving the finiteness of regularity of residue field by Avramov and Peeva in \cite{AP}, or finiteness of regularity of a module over Frobenius-like endomorphisms by Nguyen and the author in \cite{NV}. Each technique has proved to be useful in certain cases, but has not been successfully applied to resolve the problem in this class of $\Gamma.$ The proofs of Caviglia in \cite{C}, Caviglia and Conca in \cite{CC} and Tancer in \cite{T} are all different, and quite special to their setup.

\medskip

{\bf Outline of the proof of Theorem \ref{main}.} The necessity of the condition that when $d \ge 3$, $\a \neq (0, ..., 0, 2, d-2)$ is easy. In this case, the cubic relation $y^3 - xz^2$ is minimal, where $x,y,z$ correspond to the elements $(0, ..., 0, 3, d-3)$, $(0, ..., 0, 1, d-1)$, and $(0, ..., 0, d)$ respectively. Thus, we may assume from now on that $\a$ is different from $(0, ..., 0, 2, d-2)$ when $d \ge 3$. 

To prove that $K[\Gamma]$ is Koszul we distinguish two cases, according to whether $\Gamma + \Gamma = V(n,2d)$ or not. We call a subset $\Gamma$ of $V(n,d)$ $2$-full if $\Gamma + \Gamma = V(n,2d)$. In our class of $\Gamma$, given that $\a \neq (0, ..., 0, d)$, it is not $2$-full if and only if $\a = (0, ..., 0, 1, d-1)$. In this case, we will first prove that $K[\Gamma]$ is quadratic, and then prove that it has a quadratic Gr\"obner basis with respect to the grevlex order. Details are presented in section \ref{sec2}.

In the case $\Gamma$ is $2$-full, we use a theorem of Avramov and Peeva \cite[Theorem 2]{AP} stating that $R$ is Koszul if the regularity of the residue field $K$ over $R$, $\reg_RK$, is finite. To bound $\reg_R K$, we use the description of betti numbers of the residue field $K$ over the toric ring $K[\Gamma]$ as the dimensions of homology groups of certain simplicial complexes given in \cite{LS}. 

Let $t_1, ..., t_n$ be variables. For each $\alpha = (\alpha_1, ..., \alpha_n)\in \NN^n$, let $t^\alpha = t_1^{\alpha_1}\cdots t_n^{\alpha_n}$. The semigroup ring $R = K[\Gamma]$ is the subalgebra of $K[t_1, ..., t_n]$ parametrized by $x_\alpha = t^\alpha$ for $\alpha \in \Gamma$. Let $S = K[x_\alpha: \alpha  \in \Gamma]$. Also, let $I = I(\Gamma)$ be the defining ideal of $R$ in $S$. The algebra $S$, (and so is $R$) is multi-graded with $\deg x_\alpha = \alpha.$ Under this grading, the betti numbers of the residue field $K$ over $R$ can be described as follows.

For each $\lambda$ in the semigroup generated by $\Gamma$, let $\Gamma_\lambda$ be the simplicial complex of open chains from $0$ to $\lambda$ whose links are in $\Gamma$. Note that an open chain is a chain not containing the two end points. Looking at the bar resolution of the residue field $K$ over $R = K[\Gamma]$, and restricting it to certain multi-degree $\l$, one gets

\begin{Theorem}[Laudal-Sletsjoe]\label{relativeHomology} For each non-negative integer $i$ and each $\lambda$ in the semigroup generated by $\Gamma$,
$$\b_{i,\lambda}^R (K) = \dim_K \tilde H_{i-2} (\Gamma_\lambda,K).$$
\end{Theorem}

Under the standard grading, for each $\l$ in the semigroup generated by $\Gamma$, if we denote $|\l| = (\sum_{i=1}^n \l_i)/d$, then the regularity of $K$ over $R$ can be defined by 

$$\reg_R K = \sup_{i, \l}\{ |\l| - i: \b_{i, \l}^R (K) \neq 0\}.$$

Note that the simplicial complex $\Gamma_\l$ is pure of dimension $|\l|-2$. By Theorem \ref{relativeHomology}, to prove that $K[\Gamma]$ is Koszul is equivalent to proving that $\Gamma_\lambda$ has at most top dimensional homology for each $\lambda$. 

Moreover, we know that the Veronese ring $K[V(n,d)]$ is Koszul. Thus if we denote by $\Delta_\lambda$ the simplicial complex of open chains from $0$ to $\lambda$ whose links are in $V(n,d)$ then $\Delta_\lambda$ has at most top dimensional homology. 

Fix an element $\lambda$ in the semigoup generated by $\Gamma$. Our novel idea is to compare the simplicial complex $\Gamma_\lambda$ with the simplicial complex $\Delta_\lambda$. Note that $\Gamma_\l$ is obtained from $\Delta_\l$ by removing the facets whose corresponding chains have at least a link that is not in $\Gamma$. Using Theorem \ref{relativeHomology}, to prove the finiteness of regularity of $K$ over $R$, we show that we can order the set of chains so that the process of removing chains from $\Delta_\lambda$ to obtain $\Gamma_\lambda$ would not result in homology in too low dimensions.

To be more precise, let us introduce some more notation. We order the set of elements of $V(n,d)$ by lexicographic order. For any two end points $a,b$, the set of closed chains from $a$ to $b$ whose links are in $V(n,d)$ is denoted by $P(a,b)$. We order the chains in $P(a,b)$ as follows. For a closed chain $x = x^1\cdots x^n$ in $P(a,b)$, we denote by $\deg_\a x$ the number of link $x^i$ such that $x^i  = \a$, and called it $\a$-degree of $x$. For two chains $x = x^1\cdots x^n$ and $y = y^1\cdots y^n$ in $P(a,b)$, we say that $x$ is larger than $y$ if either $\deg_\a x > \deg_\a y$ or $\deg_\a x = \deg_\a y$ and $x > y$ in lexicographic order. The set of open chains from $a$ to $b$ is denoted by $OP(a,b)$. For each open chain $x$ in $OP(a,b)$, we denote by $\bar x$ its unique closed chain from $a$ to $b$ whose corresponding open chain is $x$. We then denote $\deg_\a x = \deg_\a \bar x$, and say that open chain $x$ is larger than open chain $y$ if $\bar x > \bar y$. 

The open chains in $\Delta_\lambda$ are totally ordered by this ordering, and $\Gamma_\lambda$ consists of exactly those chains in $\Delta_\lambda$ of $\a$-degrees $0$. In other words, we have 
$$\Delta_\lambda = \Gamma_\lambda \cup p^1 \cup \cdots \cup p^k$$
where $p^1 < \cdots < p^k$ are open chains from $0$ to $\lambda$ of $\a$-degrees at least $1$.

Fix a chain $p$ in $\Delta_\lambda$ of $\a$-degree at least $1$, we denote by $F_{<p}$ the set of all chains in $\Delta_\lambda$ less than $p$. We will prove in Lemma \ref{dim} that the simplicial complex $F_{<p} \cap p$ has dimension $|\l|-3$. We then prove that it has no homology in dimensions $\le |\l|-7$ by analyzing its facet structure. From this, by downward induction on $i$, we prove that $F_{<p^i}$ has no homology in dimensions $\le|\l| - 8$ for all $i$. By Theorem \ref{relativeHomology}, this implies that $\reg_R K \le 5$, see proofs in section \ref{pfmain} for more details.

\section{Proof of the main theorem in the case $\Gamma$ is not $2$-full}\label{sec2}

In this section, we prove the main theorem in the case $\Gamma$ is not $2$-full, that is $\Gamma + \Gamma \neq V(n,2d)$. In our class of $\Gamma$, given that $\a \neq (0, ..., 0, d)$, and $\a \neq (0, ..., 0, 2, d-2)$ or one of its permutation, it is not $2$-full if and only if $\a = (0,..., 0,1, d-1)$ and $d \neq 3$. In this case we prove that $R = K[\Gamma]$ is quadratic in Theorem \ref{quadratic}. The main result of this section is Theorem \ref{Koszul} where we prove that $R$ has a quadratic Gr\"obner basis in graded reverse lexicographical order. 

To prove that $I = I(\Gamma)$ is generated by quadrics, we use the description of betti numbers of $I$ in term of homology groups of certain simplicial complexes given by Bruns and Herzog in \cite{BH}. Denote by $(\Gamma)$ the semigroup generated by $\Gamma$. 

\begin{Definition}[Squarefree divisor simplicial complex]\label{divisorcomplex} For each $\l \in (\Gamma)$, let $D_\l$ be the simplicial complex on the vertex set $\Gamma$ such that $F \subseteq \Gamma$ is a face of $D_\l$ if and only if 
$$\l - \sum_{\alpha \in F} \alpha \in (\Gamma).$$
\end{Definition}

By \cite[Proposition 1.1]{BH} the betti numbers of $I$ and the homology groups of $D_\l$ are related by:
\begin{Theorem}\label{trans} For each $i$, and each $\l \in (\Gamma)$, 
$$\beta_{i,\l}^S(I(\Gamma)) = \dim_K \tilde H_i (D_\l,K).$$
\end{Theorem}

For $a, b\in \NN^n$, we write $a \ll b$ to mean $a_i \le b_i$ for all $i = 1, ..., n$. Denote $\Delta_{\ll a} = \{z\in V(n,d): z\ll a\}$ and $\Gamma_{\ll a} = \{z\in \Gamma: z\ll a\}.$ Also, we write $\supp a = \{i: a_i \neq 0\}$.

\begin{Theorem}\label{quadratic} Assume that $d \neq 3$ and that $\a = (0, ..., 0, 1, d-1)$. The algebra $K[\Gamma]$ is quadratic.
\end{Theorem}
\begin{proof} By Theorem \ref{trans}, it suffices to prove that for any $\l \in (\Gamma)$ with $|\l| = k \ge 3$, the divisor simplicial complex $D_\l$ is connected. By induction on $n$ we may assume that $\l$ has full support, i.e., $\supp \l = \{1, ..., n\}$. From the fact that the Veronese ring $K[V(n,d)]$ is quadratic, we deduce that if $\a \not \ll \l$ then $D_\l$ is connected. Thus we may further assume that $\a \ll \l$, in other words $\l_n \ge d-1$. Let $y$ be the smallest vertex of $D_\l$ in the lexicographic order. We will show that for any other vertex $x$ of $D_\l$, $x$ and $y$ are connected by a sequence of edges in $D_\l$. We will treat the cases $d = 2$ and $d \ge 4$ separatedly. We will make use of the following notation. For each $i$, $1 \le i \le n$, let $z_i = \max (x_i, y_i)$. Since $x,y \ll \l$, we have $z \ll \l$. Let $u = \l - z$. For each $v$ in $\Delta_{\ll u}$, we have $\l - (x + v)$ and $\l - (y + v)$ are in $V(n,(k-2)d)$. Since $\sum_{i=1}^n z_i  \le \sum_{i=1}^n (x_i + y_i) = 2d$, we have $\sum_{i=1}^n u_i \ge (k-2)d$.

\smallskip
\noindent
{\bf Case 1:} $d = 2$. The cases $n = 2$ and $n = 3$ are easily verified, we may assume that $n \ge 4$. Moreover, without loss of generality, we may assume that $\l_{n-1} \le \l_{n}$. Note that for any $j\ge 1$ the only elements of $V(n,2j)$ which are not in $(\Gamma)$ are of the form $(0, ..., 0, i, 2j-i)$ for some odd integer $i$. We have the following subcases: 

{\bf Case 1a:} $\l_n \ge 2$. In this case $y = (0, ..., 0, 2)$, as if $(0, ..., 0,2)$ is not a vertex of $D_\l$, then $\l = (0, ..., 0, 2) + (0, ..., 0, i, 2(k-1) -i) \notin (\Gamma)$, which is a contradiction. If $x_n + 2 \le \l_n$, i.e., $x + y \ll \l$, then $xy$ is an edge of $D_\l$, as if not, $\l - x = y + (0, ...,0, i, 2(k-2) - i) \notin (\Gamma)$, which is a contradiction. Thus we may assume that $\l_n = 2$ and $x_n = 1$. Since $x$ is a vertex of $D_\l$, 
$$\l = x + x^1 + \cdots + x^{k-1}$$
for some $x^i \in \Gamma$. Since $\l_n - x_n = 1$, there must exists $j$ so that $x^j_n = 0$. Thus $x^j y$ and $xx^j$ are edges of $D_\l$.

{\bf Case 1b:} $\l_n = 1$. In this case by our assumption $\l_{n-1} = 1$, and $y = (0, ...,0, 1, 0, 1)$. Since $x$ is a vertex of $D_\l$, 
$$\l = x + x^1 + \cdots + x^{k-1}$$
for some $x^i \in \Gamma$. Since $\l_n = 1 = x_n + \cdots + x^{k-1}_n$, there are $k-1$ elements $v^1, ..., v^{k-1}$ among the elements $x, x^1, ..., x^{k-1}$ such that $v^i_n = 0$ for all $i$. Since $\l_{n-2} = x_{n-2} + x^1_{n-2} + \cdots+x^{k-1}_{n-2} \ge 1$, there exists an $x^i$ such that $x^i_{n-2} \ge 1$. In particular, there exists at least one element $v$ among the elements $v^1, ..., v^{k-1}$ such that $v_{n-2} + 1\le \l_{n-2}$. This element $v$ is connected to $x$. Since $\l_n -y_n = 0$ and $\l - v-y\in V(n,(k-2)d)$, $yv$ is an edge of $D_\l$.

\smallskip 
\noindent 
{\bf Case 2:} $d \ge 4$. Note that for any $j\ge 1$ the only element of $V(n,jd)$ which is not in $(\Gamma)$ is $(0, ..., 0, 1, jd-1)$. There are following subcases:

{\bf Case 2a:} $\l_n \ge d$. In this case $y = (0, ..., 0, d)$ as if $(0, ..., 0, d)$ is not a vertex of $D_\l$, then $\l = y + (0, ..., 0, 1, (k-1)d-1) = (0, ..., 0, 1, kd-1)\notin (\Gamma).$ If $x_n + d \le \l_n$, i.e. $x + y \ll \l$, then $xy$ is an edge of $D_\l$, as if $xy$ is not an edge of $D_\l$, then 
$$\l - x = y + (0, ..., 0, 1, (k-2)d-1) = (0, ..., 0, 1, (k-1)d-1)\notin (\Gamma),$$
which implies that $x$ is not a vertex of $D_\l$, which is a contradiction. Thus we may assume that $\l_n < x_n +d$. In particular, $\l_n -d \le d-2$ and $x_n \ge 1$. Therefore $\a \not \ll u$, and $\sum_{i=1}^n u_i \ge (k-2)d + 1 \ge d+1$. 

For any $v \in \Delta_{\ll u}$, $vy$ is an edge, as $\l_n - y_n \le d-2$, and $\l - v - y\in V(n,(k-2)d)$. 

For any $v\in \Delta_{\ll u}$, $vx$ is not an edge of $D_\l$ if and only if $\l - x - v = (0, ..., 0, 1, (k-2)d-1)$. In other words, there exists at most one element $v \in \Gamma_{\ll u}$ such that $xv$ is not an edge of $D_\l$. Thus we may assume that $\Delta_{\ll u} = \{v\}$ and that $xv$ is not an edge of $D_\l$. Since $\sum_{i=1}^n u_i \ge d+1$, the set $\Delta_{\ll u}$ has unique element if and only if $\supp u =\{i\}$ for some $i$. Since
$$\l = x + v + (0, ..., 0, 1, (k-2)d-1)$$
and $\l_n \le 2d-2$, $k  = 3$. Therefore $\l_n = \max(x_n,y_n) = d$, and $x_n = 1$. Also, $\l_{n-1} = x_{n-1} + u_{n-1} = x_{n-1} + v_{n-1} + 1$, thus $u_{n-1} \ge 1$. Therefore, $\supp u =\{n-1\}$. In particular 
$$\l = x + (0, ..., 0, d+1, d-1),$$
and so $x$ is connected to $(0, ..., d-2, 2)$. Replacing $x$ by this element, and repeating the argument above, now $x_n = 2$, one see that there exists an element $v$ such that $x$ is connected to $v$ and $v$ is connected to $y$.

{\bf Case 2b:} $\l_n = d-1$ and $\l_{n-1} > 1$. In this case $y = (0, ..., 0, 2, d-2)$, as this is the smallest element in $\Gamma_{\ll \l}$ and $\l -y\in (\Gamma)$. For any element $v\in \Gamma$, such that $y + v \ll \l$, we have $yv$ is an edge of $D_\l$, since $\l_n - y_n = 1 < d-1$. Thus we may assume that $x + y \not \ll \l$. This implies that $\sum_{i=1}^n u_i \ge d+1$. As in case 2a, if $\Delta_{\ll u}$ has more than one element, then there exists an element $v \ll u$ so that $xv$ and $vy$ are edges of $D_\l$. Thus we may assume that $\Delta_{\ll u}$ has only one element, in particular $\supp u = \{i\}$ for some $i$. Since $u_n\le 1$, we have $i \le n-1$. In particular $\l_n = \max(x_n,y_n) = x_{n-1}= d-1$. Thus $xv$ is an edge of $D_\l$ where $v$ is the unique element in $\Delta_{\ll u}$, since $\l_n-x_n = 0$. 

{\bf Case 2c:} $\l_n = d-1$ and $\l_n = 0$. In this case, $y = (0, ..., 0, 1, 0, d-1)$, as this is the smallest element in $\Gamma_{\ll \l}$ and $\l - y\in (\Gamma)$. As in case 2b, any element $v\in \Gamma$ such that $y + v \ll \l$ is connected to $y$ by an edge. Thus we may assume that $x + y \not \ll \l$, which implies that $\sum_{i=1}^n u_i \ge d+1$. Also, we may assume that $\supp u = \{i\}$ for some element $i$. Since $u_{n-1}, u_n \le 1$, we have $i \le n-2$. Thus $x_{n-1} = \l_{n-1} = 1$. Therefore $xv$ is an edge of $D_\l$ where $v$ is the unique element in $\Delta_{\ll u}$, since $\l_{n-1} - x_{n-1} = 0$.
\end{proof}

We will prove that $K[\Gamma]$ has a quadratic Gr\"obner basis. We refer to \cite{S} for unexplained terminology about Gr\"obner basis. By abuse of notation, each element of $\Gamma$ also denotes a variable in $S = K[x_\lambda:\lambda \in \Gamma]$. Each monomial $x^1x^2\cdots x^k$ in $S$ with $x^1 \le x^2 \le ... \le x^k$ corresponds to a chain from $0$ to $x^1 + \cdots + x^k$ whose links are $x^1, ..., x^k$. Recall from the introduction that the set of chains $P(a,b)$ with fixed endpoints $a,b$ is totally ordered. We say that a chain $x^1\cdots x^k$ is minimal if it is the minimal chain in $P(0, x^1 + \cdots + x^k)$. For a monomial $x^1 \cdots x^k$, which we think of as a chain, we write $\min(x^1\cdots x^k)$ for the monomial $y^1\cdots y^k$ such that $y^1\cdots y^k$ is the minimal chain in $P(0,x^1 + \cdots + x^k)$.

The following simple observations where $\a$ is an arbitrary element of $V(n,d)$ will be useful when dealing with minimal chains and will be used in the proof of the next lemma. These will also be useful in later section.

\begin{Fact}\label{support} Let $m\in \NN^n$ such that $\Gamma_{\ll m}$ is non-empty. Let $x$ be the smallest element in $\Gamma_{\ll m}$. Let $i$ be the index such that $\sum_{j > i} m_j < d$ while $\sum_{j\ge i} m_j \ge d$. The smallest element in $\Delta_{\ll m}$ is $s = (0, ..., 0, s_i, m_{i+1}, ..., m_n)$, where $s_i = d - (m_{i+1} + \cdots + m_n)$. If $s\neq \a$, then $x = s$. If $s = \a$ and $m_i > s_i$, then $x = (0, ..., 0, s_i+1, m_{i+1}-1, m_{i+2}, ..., m_n)$. Finally, if $s = \a$, and $m_i = s_i$, then $x= (0, ..., 0,1,0, ...0, s_i-1, m_{i+1}, ..., m_n)$ where $1$ is at position $j$, which is the largest index less than $i$ such that $m_j > 0$. 
\end{Fact}

\begin{Fact}\label{nonminimalx1} Let $x^1x^2$ be a minimal chain of $\a$-degree $0$. Let $y^1$ be the minimal element in the set $\Gamma_{\ll x^1+ x^2}$. Let $y^2 = x^1 + x^2 - y^1$. The only obstruction for the chain $y^1y^2$ not smaller than $x^1x^2$ is that $y^2 = \a$. In particular, if $x^1 > y^1$, then $x^1$ is next to minimal in the set $\Gamma_{\ll x^1+x^2}$ and $x^2 = \a + y^1-x^1$.
\end{Fact}

\begin{Lemma}\label{Lm1} Assume that $\a = (0, ..., 0, 1, d-1)$ and $d \neq 3$. Let $x,y, z$ be elements in $\Gamma$. If $xy$, $yz$ and $xz$ are minimal, then $xyz$ is a minimal chain.
\end{Lemma}
\begin{proof} The case $d = 2$ is easily verified, we assume that $d \ge 4$. Let $m = x + y + z$. It suffices to show that $x$ is the smallest element in $\Gamma_{\ll m}$. Let $u = x + y$ and $v = x + z$. If $u_n \ge d$, then by Fact \ref{support} and Fact \ref{nonminimalx1}, and $u \neq (0, ..., 1, 2d-1)$, we have $x = (0, ..., 0, d)$, which is the smallest element in $\Gamma$. Similarly, if $v_n \ge d$, $x$ is the smallest element in $\Gamma$. Thus we may assume that $u_n, v_n \le d-1$. 

If $u_n \le d-2$, by Fact \ref{support} and Fact \ref{nonminimalx1}, we have $x = (0, ...,0, s_i, u_{i+1}, ..., u_n)$ where $i$ is the index such that $u_{i+1} + \cdots + u_n < d$, while $u_i + \cdots+ u_n \ge d$ and $s_i = d - (u_{i+1} + \cdots + u_n).$ Since $u_n \le n-2$, we have $i \le n-1$. In particular $y_{i+1} = y_{i+2} = \cdots = y_n = 0$. Together with Fact \ref{support}, Fact \ref{nonminimalx1} and the fact that $yz$ is minimal, we have $z_{i+1} = \cdots = z_n = 0$. In particular $x$ is the smallest element in $\Gamma_{\ll m}$. 

Thus we may assume that $u_n = d-1$. This implies that $v_n = d-1$. If $u_{n-1} \ge 2$, by Fact \ref{support} and Fact \ref{nonminimalx1}, we have $x = (0, ...,0, 2, d-2)$. This implies that $y_n = z_n = 1$. By Fact \ref{support} and Fact \ref{nonminimalx1}, $yz$ is not minimal, which is a contradiction.

Finally, assume that $u_n = d-1$ and $u_{n-1} = 1$. By Fact \ref{support} and Fact \ref{nonminimalx1}, we have $x = (0, ..., 0, 1, 0, ..., 0, d-1)$, where $1$ is in position $j$, the largest index less than $n-1$ such that $u_j > 0$. Moreover, in this case $y_l = 0$ for all $j < l\le n-2$, and $y_n = 0$. Since $xz$ is minimal, we have $z_{n-1} \le 1$, $z_l = 0$ for all $j \le n \le n-2$ and $z_n = 0$. Since $yz$ is minimal, $z_{n-1} = 0$. In particular, $x$ is also the minimal element in $\Gamma_{\ll m}$.
\end{proof} 

Let $\prec$ be the grevlex order on $S = K[x_\lambda: \lambda \in \Gamma]$. Let $xy$ with $x \le y$ be a non-minimal quadratic binomial. Assume that $\min (xy) = zt$. Since $z < x$, and $z + t = x + y$, $z < x \le y < t$. In particular $\min(xy) \prec xy$.

\begin{Theorem}\label{Koszul} Assume that $d \neq 3$ and $\a = (0, ..., 0, 1, d-1)$. With respect to the grevlex order $\prec$ on $S = K[x_\lambda: \lambda \in \Gamma]$, the set of quadratic binomials 
$$\mathcal G = \{x y - zt: xy \text{ is not minimal and } \min(xy) = zt \}$$
is a Gr\"obner basis for $I(\Gamma)$. As a consequence, the algebra $K[\Gamma]$ is Koszul.
\end{Theorem}
\begin{proof}By Theorem \ref{quadratic}, $\mathcal G$ is a minimal generating system for $I(\Gamma)$. To show that $\mathcal G$ is a Gr\"obner basis of $I$, it suffices by Buchberger's criterion to show that any cubic monomial $xyz$ reduces to $\min(xyz)$ by $\mathcal G$. 

By the choice of term order $\prec$, if any of the monomials $xy, yz, xz$ is not minimal we replace it by its minimal, then we get a smaller monomial in the equivalent class of $xyz$ modulo $\mathcal G$. Since the number of monomials in this equivalent class is finite, this procedure stops. When it stops, one gets $xy, yz$ and $xz$ are minimal. By Lemma \ref{Lm1}, $xyz$ is minimal.
\end{proof}

\begin{Remark} In general, Lemma \ref{Lm1} does not hold if $\a \neq (0, ..., 1, d-1)$. For example, consider the case of the classical pinched Veronese where $n=d=3$ and $\a = (1,1,1)$. Let $x = (0,1,2)$, $y = (1,0,2)$ and $z = (3,0,0)$. The chains $xy, xz$ and $yz$ are minimal but $xyz$ is not. As $x+ y+z = u+v+w$, with $u = (0,0,3)$, $v = (2,0,1)$ and $w = (2,1,0)$ and $uvw < xyz$. 

Theorem \ref{Koszul} still holds if $\Gamma$ is replaced by $V(n,d)$. In other words, it gives a new quadratic Gr\"obner basis of the Veronese ring $K[V(n,d)]$ (see \cite[Theorem 14.2]{S} for the classcial quadratic Gr\"obner basis for $K[V(n,d)]$).
\end{Remark}

\section{Proof of the main theorem in the case $\Gamma$ is $2$-full}{\label{pfmain}

In this section, we will give a proof of the main theorem in the case $\Gamma$ is $2$-full. Throughout this section, we may assume that $\a$ is different from $(0, ..., 0, d)$, $(0,..., 0, 1, d-1)$ and $(0, ..., 0,2, d-2)$.  In particular, $\supp \a \supseteq \{n-1, n\}$.

Let us first recall some notation from the introduction. For each $\l$ in the semigroup $(\Gamma)$, $\Gamma_\l$ is the simplicial complex of open chains whose links are in $\Gamma$, while $\Delta_\l$ is the simplicial complex of open chains whose links are in $V(n,d)$. To show that $K[\Gamma]$ is Koszul, we will show that for each $\l$, with $|\l| \ge 7$, the simplicial complex $\Gamma_\lambda$ has no homology in dimensions $\le |\l| - 8$. To accomplish this, we will show in Lemma \ref{dim} that for each chain $p$ in $\Delta_\lambda$ of degree at least $1$, the simplicial complex $F_{<p} \cap p$ has dimension equal to $\dim \Delta_\l - 1 = |\l| - 3$. By analyzing the facet structure of $F_{<p} \cap p$, we will show that it has no homology in dimensions $\le |\l| - 7$. 

The following property and notation will be used frequently in the proofs of the following lemmas. For an integer $n$, $[n]$ denote the set of elements $\{1, ..., n\}$. By the $2$-fullness, if $a \ll b$ are elements in the semigroup generated by $V(n,d)$ and $|b| - |a| \ge 2$, then there is a closed chain from $a$ to $b$ whose links are in $\Gamma$.

Fix an open chain $p$ in $\Delta_\l$ which is not in $\Gamma_\l$. The corresponding closed chain from $0$ to $\lambda$ is denoted by $\bar p$. We label the nodes of $p$ by $1, ..., n$. The label $0$ stands for the origin $0$, and $n+1$ stands for $\l$. For any consecutive set of indices $L = \{i, ..., i+j\}$, denote $\bar L$ the subchain of the closed chain $\bar p$ going from node $i-1$ to node $i+j+1$. A chain from node $i$ to node $j$ is either denoted by its nodes $(i)(i+1)\cdots (j)$ or by its links $x^1 \cdots x^l$. Note that $\Delta_\l$ and $\Gamma_\l$ are both pure simplicial complex of dimension $n-1$, and $[n]$ is a facet of $\Delta_\l$. 

The proof of the main theorem in the case $\Gamma$ is $2$-full relies on the following series of lemmas. 

\begin{Lemma}\label{dim} $\dim F_{<p} \cap p = \dim \Delta_\l - 1$. 
\end{Lemma}
\begin{proof}Since $p$ is not in $\Gamma_\l$, the $\a$-degree of $p$ is at least $1$. Thus, there exists an index $i$, $0 \le i \le n$ such that the link from $i$ to $i+1$ is equal to $\a$. If $i \le n-1$, then by the $2$-fullness condition, there is a closed chain going from $i$ to $i+2$ whose links are in $\Gamma$. Denote this chain by $(i)(l) (i+2)$. Let $\bar q$ be the closed chain $(0)\cdots (i-1)(i)(l)(i+2)(i+3)\cdots (n+1)$ in $P(0, \l)$. Let $q$ be the corresponding open chain, then $\deg_\a q  \le \deg_\a p - 1$, thus $q < p$, and $q \cap p = [n] \setminus \{i+1\}$. Now if $i \ge 1$, there is a close chain going from $i-1$ to $i+1$ whose links are in $\Gamma$. If we denote this chain by $(i-1)(l)(i+1)$, and let $\bar q$ be the closed chain $(0)(1)\cdots (i-1)(l)(i+1)\cdots (n+1)$, then $q < p$ and $p\cap q = [n] \setminus \{i\}.$ Thus there is at least one facet of $F_{< p}\cap p$ of the form $[n] \setminus \{i\}$ for some $i$. Consequently, $\dim F_{<p} \cap p = \dim \Delta_\l -1$. 
\end{proof}

\begin{Lemma}\label{facet1} A facet of $F_{<p}\cap p$ is of the form $[n] \setminus L$, where $L = \{i, i+1, ..., i+j\}$ is a set of consecutive indices.
\end{Lemma}
\begin{proof} Let $F$ be a facet of $F_{<p} \cap p$. Then $F = p \cap q$ for some chain $q < p$. Assume that $F = [n] \setminus (L_1 \cup L_2 \cup \cdots \cup L_k)$ where $L_1, ..., L_k$ are disjoint consecutive sets of indices. Assume by contradiction that $k \ge 2$. From the proof of Lemma \ref{dim}, if $\deg_\a \bar L_i \ge 1$ for some $i$, then there is a facet of $F_{<p} \cap p$ of dimension $n-2$ containing $F$, which is a contradiction. Therefore $\deg_\a \bar L_i = 0$ for all $i$.

Now we claim that for every $i$, the chain $\bar L_i$ is a minimal chain. Assume that $\bar L_i$ is not minimal for some $i$. Let $\bar L_i = (l)(l+1)\cdots (k)$, and let $\bar L_i' = (l)(x)(x+1)\cdots(y)(k)$ be a smaller chain. Let $q'$ be the open chain whose corresponding closed chain is $\bar q' = (0)\cdots (l)(x)(x+1)\cdots (y)(k) \cdots (n+1)$. Since $\bar L_i' < \bar L_i$, we have $\bar q' < \bar p$, thus $q' < p$. Moreover, $q' \cap p = [n] \setminus L_i$ containing $F$, which is a contradiction. 

Thus $L_i$ is minimal for all $i$. Now, note that $\deg_\a \bar q = \deg_\a \bar p$, as $\deg_\a \bar L_i = 0$ for all $i$. Since $\bar L_i$ are minimal, $q \ge p$, which is a contradiction.
\end{proof}

The following simple observations will be useful in the sequences. First, let us recall the following notation introduced in section \ref{sec2}. For each $m \in \NN^n$, denote $\Gamma_{\ll m} = \{\l \in \Gamma: \lambda \ll m\}$, and $\Delta_{\ll m} = \{\l\in V(n,d): \l \ll m\}.$

\begin{Fact}\label{minimalpath} Let $x^1\cdots x^k$ be a minimal chain of $\a$-degree $0$ and $k \ge 3$. Let $m = x^1 + \cdots +x^k$. Denote by $y^1$ the minimal element in the set $\Gamma_{\ll m}$. By the $2$-fullness condition, there exist $y^2, ..., y^k$ in $\Gamma$ such that $y^2 + \cdots + y^k = x^1 + \cdots + x^k - y^1$. Since $x^1 \cdots x^k$ is a minimal chain, $y^1 \cdots y^k \ge x^1 \cdots x^k$, thus $y^1 \ge x^1$. In other words, $x^1$ is the minimal element in the set $\Gamma_{\ll m}$.
\end{Fact}

\begin{Fact}\label{orderofL} Let $F = [n] \setminus L$ be a facet of $F_{<p} \cap p$. Assume that $\bar L = x^1 \cdots x^k$ and $k \ge 3$. From the proof of Lemma \ref{facet1}, we have $\deg_\a \bar L =0$ and $\bar L$ is not minimal.  

Assume that a proper subchain $x^l \cdots x^t$ of $\bar L$ is not minimal. Let $x^l \cdots x^k = (i+l-1) \cdots (i+t)$. Let $(i+l-1)(s)(s+1)\cdots (i+t)$ be the minimal chain whose endpoints are $(i+l-1)$ and $(i+t)$. Let $\bar q = (0)\cdots (i+l-1)(s)(s+1)\cdots (i+t)\cdots(n+1)$, then $q < p$, but $q \cap p \supsetneq F$, which is a contradiction. Therefore any proper subchain of $\bar L$ is minimal. In particular $x^1 \le x^2 \le \cdots \le x^k$.
\end{Fact}

We are now ready to analyze in more detail the structure of the simplicial complex $F_{<p} \cap p$. 

\begin{Lemma}\label{facet2} Let $F = [n] \setminus L$ be a facet of $F_{<p}\cap p$. Then $|L| \le 2$. 
\end{Lemma}
\begin{proof} Let $F = [n] \setminus L$ where $L = \{i+1, ..., i+k-1\}$ be a facet of $F_{<p} \cap p$. Assume that $k\ge 4$, and $\bar L = x^1\cdots x^k$. Let $m = x^1 + \cdots + x^k = y^1 + \cdots + y^k$. Let $u = x^1 + \cdots + x^{k-1}$ and $v = x^2 + \cdots + x^k$. By Fact \ref{minimalpath} and Fact \ref{orderofL}, $x^1$ is the minimal element in $\Gamma_{\ll u}$ and $x^2$ is the minimal element in $\Gamma_{\ll v}$. Moreover, by Fact \ref{orderofL}, $x^1\cdots x^k$ is not a minimal chain, thus $x^1$ is not the smallest element in $\Gamma_{\ll m}$. 

By Fact \ref{support}, there are following possibilities for $x^1$:

{\bf Case 1:} $x^1 = (0, ..., s_i, u_{i+1}, ..., u_n)$ for some $i$. Since $x^1$ is not minimal in $\Gamma_{\ll m}$, $i < n$. Thus $x^2_j = ... = x^{k-1}_j = 0$ for all $j \ge i+1$. If $x^k_{i+1} = ... = x^k_n = 0$, then $x^1$ is also the smallest element in $\Gamma_{\ll m}$ which is a contradiction. Thus $x^k_j > 0$ for some $j \ge i+1$. Moreover, $x^2$ is the smallest element in $\Gamma_{\ll v}$, by Fact \ref{support}, this implies that $\a$ is the smallest element in $\Delta_{\ll v}$. Let $j$ be the index such that $\sum_{t > j} v_t < d$ while $\sum_{t \ge j} v_t \ge d$. By Fact \ref{support}, we have $\a = (0, ..., 0, a_j, v_{j+1}, ..., v_n)$. In particular, $j <n$, and $v_t > 0$ for all $t \ge j$. Since $x^2$ is the minimal element in $\Gamma_{\ll v}$, by Fact \ref{support} and the fact that $x^2_n = 0$, we have $j = n-1$, and $v_n = 1$. This implies that $\a = (0, ..., 0, 1, 1)$, since $a_1 \le a_2 \le \cdots \le a_{n-1} \le a_n$, which is a contradicition since $d \ge 3$.

{\bf Case 2:} $x^1 = (0, ..., 0, s_i+1, u_{i+1}-1, u_{i+2}, ..., u_n)$ for some $i$. By Fact \ref{support}, this is the case if and only if $\a = (0,...,0,s_i, u_{i+1}, ..., u_n)$, and $u_i > s_i$. Since $\supp \a$ contains $n-1$, we have $i < n$. Moreover, we have $x^2_t = ... = x^{k-1}_t = 0$ for all $t \ge i+2$ and $x^2_{i+1} + \cdots + x^{k-1}_{i+1} = 1$. 

If $\a \not \ll v$, then by Fact \ref{support} and the fact that $x^2$ is the minimal element in $\Gamma_{\ll v}$, we have $x^2 = (0, ..., 0,x^2_j, v_{j+1}, ..., v_n)$ for some $j$. Since $x^2_t = 0$ for all $t \ge i+2$, and $x^2_{i+1} \le 1$, this implies that $j \le i$. Therefore $x^k_t = 0$ for all $t \ge i+1$. In particular, $x^1$ is the minimal element in $\Gamma_{\ll m}$, which is a contradiction. 

Thus we may assume that $\a = (0, ..., 0, s_i, v_{i+1}, ..., v_n)$. Since $x^2$ is minimal in $\Gamma_{\ll v}$, by Fact \ref{support} and the fact that $x^2_{i+2} = 0$, this implies that $i = n-1$. Therefore, we have $\a = (0,...,0,s_i, u_n) = (0,...,0,s_i, v_n)$. If $x^k_n = 0$, then $x^1 = (0,...,0,s_{n-1}+1, u_n-1)$ is the smallest element in $\Gamma_{\ll m}$, which is a contradiction. Thus we may assume that $x^k_n > 0$. By Fact \ref{support}, and the fact that $x^2$ is minimal in $\Gamma_{\ll v}$, this implies that $x^2 = (0,...,0,s_i+1,v_n-1)$. Since $x^2_n \le 1$, we have $v_n \le 2$. Since $d \ge 3$ and $a_{n-1} \le a_n$, this implies that $\a = (0, ..., 0, 1, 2)$ or $\a = (0, ..., 0, 2, 2)$. But they are not possible, since $\a$ is not of the form $(0, ...,0, 2, d-2)$ or one of its permutations. 

{\bf Case 3:} $x^1 = (0, ..., 1, 0, ..., 0, u_i - 1, u_{i+1}, ..., u_n)$ for some $i$, and $1$ is in position $j < i$. By Fact \ref{support}, this is the case if and only if $x^2_l= ...= x^{k-1}_l = 0$ for all $l \ge i+1$ and all $j+1 \le l < i$, and $\a = (0,...,0,u_i, u_{i+1}, ..., u_n).$ Moreover $x^2_i + \cdots + x^{k-1}_i = 1$. Since $\supp \a$ contains $n-1$, we have $i \le n-1$. In particular $x^2_n = 0$. 

If $x^k_n = 0$, then $v_n = 0$. Therefore $\a \not \ll v$. By Fact \ref{support} and the fact that $x^2$ is minimal in $\Gamma_{\ll v}$, this implies that $x^k_j = 0$ for all $l \ge i$ and all $j+1 \le l < i$. In particular $x^1$ is the smallest element in $\Gamma_{\ll m}$, which is a contradiction. 

Thus we may assume that $x^k_n > 0$. By Fact \ref{support}, the fact that $x^2$ is minimal in $\Gamma_{\ll v}$ and the fact that $x^2_n = 0$, this implies that $i = n-1$. In particular, we have $\a = (0, ..., 0, a_{n-1}, v_n)$. Moreover, $x^2 = (0,..., 0, a_{n-1} +1, v_n-1)$. Therefore $v_n = 1$. Since $a_{n-1} \le a_n$, we have $\a = (0, ..., 0, 1, 1)$, which is a contradiction.
\end{proof}

\begin{Lemma}\label{facet3} If $[n] \setminus \{i, i+1\}$ is a facet of $F_{<p} \cap p$, then neither $[n] \setminus \{i+1,i+2\}$ nor $[n] \setminus \{i+2,i+3\}$ are facets of $F_{<p} \cap p$. 
\end{Lemma}
\begin{proof} We proceed as in the proof of Lemma \ref{facet2}. Assume that $F = [n]\setminus L$ is a facet of $F_{<p}\cap p$, where $L = \{i, i+1\}$. Moreover, assume that $F = p \cap q$. Let $\bar L = x^1x^2x^3$. By Fact \ref{orderofL}, $\bar L$ has the property that $\bar L$ is not a minimal chain but $x^1x^2$ and $x^2x^3$ are minimal, and moreover $x^1 \le  x^2 \le x^3$. Let $x^1 + x^2 + x^3 = m$. Since $x^1x^2x^3$ is not minimal, $x^1$ is not the smallest element in $\Gamma_{\ll m}$. Let $u = x^1 + x^2$, $v = x^2 + x^3$. There are three cases:

\noindent 
{\bf Case 1:} $x^1$ is the smallest element in $\Gamma_{\ll u}$ and $x^2$ is the smallest element in $\Gamma_{\ll v}$. By Fact \ref{support}, there are following possiblities for $x^1$:

{\bf Case 1a:} $x^1 = (0, ..., 0, s_i, u_{i+1}, ..., u_n)$. Since $x^1$ is not the smallest element in $\Gamma_{\ll m}$, we have $i < n$. Therefore $x^2_{i+1} = ... = x^2_n = 0$. Moreover, if $x^3_{i+1} = ... = x^3_n = 0$, then $x^1$ is also the smallest element in $\Gamma_{\ll m}$, which is a contradiction. Therefore, $x^3_j > 0$ for some $j \ge i+1$. By Fact \ref{support}, the fact that $x^2$ is minimal in $\Gamma_{\ll v}$, and $x^2_n = 0$, this implies that $j = i+1 = n$, and $\a = (0, ..., 1, 1)$, which is a contradiction since $d \ge 3$.

{\bf Case 1b:} $x^1 = (0, ..., s_i+1, u_{i+1} - 1, u_{i+2}, ..., u_n)$. By Fact \ref{support}, this implies $\a = (0, ...,0, s_i, u_{i+1}, ..., u_n)$, $x^2_j = 0$ for all $j \ge i+2$, and $x^2_{i+1} = 1$. 

If $x^3_j = 0$ for all $j \ge i+1$, then $x^1$ is the minimal element in $\Gamma_{\ll m}$, which is a contradiction. Therefore $x^3_j > 0$ for some $j \ge i+1$. 

If $x^3_{j} > 0$ for some $j \ge i+2$, then by Fact \ref{support}, the fact that $x^2$ is minimal in $\Gamma_{\ll v}$, and $x^2_j = 0$ for all $j \ge i+2$, this implies that $j +2 = n$, and $\a = (0,...,a_{n-1}, v_n)$. This is a contradicition, since $i \le n-2$ and $\a = (0, ..., s_i, u_{i+1}, ..., u_n)$. 

If $x^3_{i+1} > 0$, then by Fact \ref{support}, the fact that $x^2$ is minimal in $\Gamma_{\ll v}$, and the fact that $x^2_j = 0$ for all $j \ge i+2$ and $x^2_{i+1} = 1$, this implies that $\a = (0, ...,0, s_i, v_{i+1}, ..., v_n)$, and $x^2 = (0, ..., 0, s_i+1, v_{i+1}-1, ..., v_n)$. Since $x^2_{i+2} = 0$, we have $i = n-1$. Since $x^2_n = x^2_{i+1} = 1$, we have $v_n = 2$. Therefore $\a$ is of the form $(0, ..., 0, 1, 2)$ or $(0, ..., 0, 2, 2)$, which is a contradiction.

{\bf Case 1c:} $x^1 = (0, ..., 1,0, ...,0, s_i - 1, u_{i+1}, ..., u_n)$ where $1$ is in position $j < i$. By Fact \ref{support}, this implies that $\a = (0, ...,0, u_i, u_{i+1}, ..., u_n)$. Since $\supp \a$ contains $n-1$, we have $i \le n-1$. Moreover, we have $x^2_i = 1$ and $x^2_t = 0$ for all $t \ge i+1$, and all $j+1 \le t < i$. 

If $x^3_n = 0$, then $\a \not \ll v$. By Fact \ref{support}, and the fact that $x^2$ is the smallest element in $\Gamma_{\ll v}$, we have $x^3_t = 0$ for all $t \ge i+1$, and all $j+1 \le t < i$. In particular, $x^1$ is the smallest element in $\Gamma_{\ll m}$, which is a contradiction.

Therefore $x^3_n > 0$. Since $x^2_n = 0$, by Fact \ref{support} and the fact that $x^2$ is the minimal element in $\Gamma_{\ll v}$, we have $i  = n-1$, and $\a = (0, ..., 0,u_{n-1}, v_n)$. Moreover, since $x^2_n = 0$, we have $v_n = 1$, which further implies that $\a = (0, ..., 0, 1, 1)$ which is a contradiction.

{\bf Case 2:} $x^1$ is the smallest element in $\Gamma_{\ll u}$ and $x^2$ is not the smallest element in $\Gamma_{\ll v}$. By Fact \ref{nonminimalx1}, this implies that $x^3 = \a + y - x^2$, where $y$ and $x^2$ are the minimal and next to minimal elements in $\Gamma_{\ll v}$. Since $x^3 \ge x^2$, this implies that $\a > x^2 > y$. 

If $v_n < d$, by Fact \ref{support}, the smallest element $y$ in $\Gamma_{\ll v}$ is of the form $y = (0, ..., s_j, v_{j+1}, ..., v_n)$ for some $j < n$. Nevertheless, since $v = \a + y$, thus $v_n = a_n + y_n$. Therefore $a_n = 0$, which is a contradiction.

If $v_n \ge d$, since $\a \ll v$, we have $v_{n-1} \ge 1$. Therefore $x^2 = (0, ..., 0, 1, d-1)$. Since $x^1$ is the smallest element in $\Gamma_{\ll u}$, we have $x^1 = (0, ..., 0, d)$ which is also the smallest element in $\Gamma_{\ll m}$, which is a contradiction.

{\bf Case 3:} $x^1$ is not the smallest element in $\Gamma_{\ll u}$. By Fact \ref{nonminimalx1}, and the fact that $x^1x^2$ is minimal, we have $x^2 = \a + y-x^1$, where $y$ and $x^1$ are the minimal and next to minimal elements in $\Gamma_{\ll u}$. Since $x^2 \ge x^1$, this implies $\a > x^1 > y$. 

If $u_n < d$, by Fact \ref{support}, the smallest element $y$ in $\Gamma_{\ll u}$ is of the form $y = (0, ..., s_i, u_{i+1}, ..., u_n)$ for some $i < n$. Since $u = \a + y$, we have $a_n = 0$, which is a contradiction. 

If $u_n \ge d$, since $\a \ll u$, we have $u_{n-1} \ge 1$. Therefore, $x^1 = (0, ...,0,1, d-1)$. In this case, we have $x^2 = \a + (0, ..., 0, -1, 1)$. Since $\a \neq (0, ..., 2, d-2)$, we have $x^2 > x^1$, and finally $x_3 \ge x_2$. 

In all cases, the smallest element $x^1$ has to be of the form $(0, ...,0, 1, d-1)$, and $x^3 \ge x^2 > x^1$. Therefore, if $(i-1)(i)$ is $x^1$, then $(i)(i+1)$ and $(i+1)(i+2)$ cannot be of this form. The lemma follows.
\end{proof}

\begin{Lemma}\label{homology1} Let $n \ge 3$. Let $\mathcal F$ be a non-empty simplicial complex on $[n]$ of dimension $n-3$ whose facets are of the form $[n]\setminus \{i,i+1\}$ for some $i$. Furthermore, assume that if $[n]\setminus \{i, i+1\}$ is a facet of $\mathcal F$, then $[n] \setminus \{i+1, i+2\}$, and $[n] \setminus \{i+2,i+3\}$ are not facets of $F$. Then $\mathcal F$ has trivial homology groups.
\end{Lemma}
\begin{proof} If $[n]\setminus \{1,2\}$ is not a facet of $\mathcal F$, then all facets of $\mathcal F$ contain $1$. If $[n] \setminus \{1,2\}$ is a facet of $\mathcal F$, then $[n] \setminus \{2,3\}$ and $[n]\setminus \{3,4\}$ are not facets of $\mathcal F$. Thus all facets of $\mathcal F$ contain $3$. In any case $\mathcal F$ is a cone, thus has trivial homology groups.
\end{proof}

\begin{Lemma}\label{homology} Let $n \ge 4$. Let $\mathcal F$ be a non-empty simplicial complex on $[n]$ of dimension $n - 2$ whose facets are of the form $[n] \setminus \{i\}$ or $[n]\setminus \{i,i+1\}$ for some $i$. Furthermore, assume that if $[n]\setminus \{i, i+1\}$ is a facet of $\mathcal F$, then $[n] \setminus \{i+1, i+2\}$, and $[n] \setminus \{i+2,i+3\}$ are not facets of $F$. Then $\mathcal F$ has no homology in dimensions $\le n-5$.
\end{Lemma}
\begin{proof} We prove by induction on $n$. The cases $n = 4$ and $n = 5$ are trivial. We may assume that $n \ge 6$. 

Let $\mathcal F$ be a non-empty simplicial complex on $[n]$ satisfying the condition of the Lemma, where $n \ge 6$. We will prove by induction on the number of facets of $\mathcal F$ that it has no homology in dimension $\le n-5$. If it has only one facet, the statement is trivial.

Let $F$ be a facet of $\mathcal F$ of dimension $n-2$. Write $\mathcal F = F \cup \mathcal G$, where $\mathcal G$ is a non-empty simplicial complex on $[n]$. We have either $\mathcal G$ satisfies the condition of the Lemma and that $\mathcal G$ has fewer facets than $\mathcal F$ or $\mathcal G$ satisfies the condition of the Lemma \ref{homology1}. By induction and Lemma \ref{homology1}, $\mathcal G$ has no homology in dimensions $\le n-5$. Applying the Mayer-Vietoris sequence, for each $i$, we get an exact sequence
$$H_i(\mathcal G) \to H_i (\mathcal F) \to H_{i-1}(F \cap \mathcal G).$$
Fix $i \le n-5$. By induction, the first term in the exact sequence is zero. Moreover, $F \cap \mathcal G$ is a non-empty simplicial complex on a set of $n-1$ vertices satisfying the condition of either the Lemma \ref{homology} or the Lemma \ref{homology1}. Since $i-1\le n-6$, by induction on $n$, the last term is zero. Therefore $H_i(\mathcal F) = 0$. 
\end{proof}

\begin{proof}[Proof of Theorem \ref{main}] By Theorem \ref{Koszul}, we may assume that $\a$ is different from $(0, ...,0, 2, d-2)$ and that $\Gamma$ is $2$-full. Assume that $|\l| \ge 7$, i.e., $\Gamma_\l$ is pure simplicial complex of dimension $n-1 = |\l|-2 \ge 5$. We will prove by downward induction on $i$ that $F_{<p^i}$ has no homology in dimensions $\le n-6$. 

When $i > k$, then $F_{<p^i}$ is $\Delta_\l$ which is known to have at most homology in dimension $n-1$. Assume that it is true for $i > 1$. Applying the Mayer-Vietoris sequence, for any $j\le n-6$ we have an exact sequence
$$ H_{j}(F_{<p^i} \cap p^i) \to H_j(F_{<p^i}) \to H_j(F_{<p^{i+1}}).$$
By induction the last term is zero. By Lemma \ref{homology}, and property of $F_{<p^i} \cap p^i$, the first term is zero. Therefore $H_j(F_{<p^i}) = 0$. In particular, when $i = 1$, $F_{<p^1} = \Gamma_{\lambda}$ has no homology in dimensions $\le n-6$. By Theorem \ref{relativeHomology}, $\reg_R K \le 5$. By \cite[Theorem 2]{AP}, $R$ is Koszul.
\end{proof}

\section*{Acknowledgements}
I would like to thank my advisor David Eisenbud for useful conversations and comments on earlier drafts of the paper.

\end{document}